\documentclass{scrartcl}
\usepackage{amsmath}
\usepackage{amsfonts}
\usepackage{amssymb}
\usepackage{hyperref}
\usepackage{cleveref}
\usepackage{color}
\newcommand{\bbbc}{\mathbb{C}}
\newcommand{\bbbn}{\mathbb{N}}
\newcommand{\bbbr}{\mathbb{R}}
\newcommand{\bbbz}{\mathbb{Z}}
\newtheorem{theorem}{Theorem}
\newtheorem{corollary}[theorem]{Corollary}
\newtheorem{lemma}[theorem]{Lemma}
\newtheorem{example}[theorem]{Example}
\newtheorem{remark}[theorem]{Remark}
\newenvironment{proof}{\noindent\emph{Proof:}}{\quad $\Box$ \smallskip}
\newcommand{\Cin}{C_\text{in}}
\newcommand{\Cst}{C_\text{st}}
\newcommand{\Cax}{C_\text{ap}}
\newcommand{\Cca}{C_\text{ca}}
\newcommand{\Cos}{C_\text{os}}
\allowdisplaybreaks
\begin{document}
\title{On iterated interpolation}
\author{Steffen B\"orm}
\date{\today}
\maketitle
\begin{abstract}
\noindent
Matrices resulting from the discretization of a kernel function,
e.g., in the context of integral equations or sampling probability
distributions, can frequently be approximated by interpolation.
In order to improve the efficiency, a multi-level approach can be
employed that involves interpolating the kernel functions and
its approximations multiple times.

This article presents a new approach to analyze the error incurred
by these iterated interpolation procedures that is considerably more
elegant than its predecessors and allows us to treat not only the kernel
function itself, but also its derivatives.
\end{abstract}


\medskip
\emph{Acknowledgment:} I owe a debt of gratitude to
Jens Markus Melenk (TU Wien) and Markus Nie\ss{} (CAU Kiel) for discussions
that form the foundation of this article.

\section{Introduction}

Let us consider a model problem from astrophysics:
we have $n$ bodies with masses $m_1,\ldots,m_n$ at points
$x_1,\ldots,x_n$ in space, and we want to evaluate the
resulting gravitational forces $f_1,\ldots,f_n$ in $n$ points
$y_1,\ldots,y_n$.
Newton's law yields
\begin{align*}
  f_i &= \sum_{j=1}^n m_j\, g(y_i,x_j), &
  g(y,x) &:= c \frac{x-y}{\|x-y\|^3}, &
  &\text{ for all } i\in[1:n],
\end{align*}
where $c$ is the gravitational constant, $\|z\|$ is the
Euclidean norm, and $[1:n] := \{ 1,\ldots,n \}$.
Direct evaluation of all $f_i$ would require
$\mathcal{O}(n^2)$ operations and is therefore unattractive or
even practically impossible if $n$ is large.
In order to evaluate the forces efficiently, we can approximate
the function $g$ by sums of tensor products, i.e.,
\begin{align}\label{eq:degenerate}
  g(y,x) &\approx \sum_{\nu=1}^k a_\nu(y) b_\nu(x) &
  &\text{ for all } x\in\sigma,\ y\in\tau,
\end{align}
with suitable functions $a_1,\ldots,a_k$ on a domain $\tau$
and $b_1,\ldots,b_k$ on a domain $\sigma$.
If all $y_i$ are in $\tau$ and all $x_j$ are in $\sigma$, we
obtain
\begin{align*}
  f_i &\approx \sum_{j=1}^n m_j
               \sum_{\nu=1}^k a_\nu(y_i) b_\nu(x_j)
       = \sum_{\nu=1}^k a_\nu(y_i)
         \underbrace{\sum_{j=1}^n m_j b_\nu(x_j)}_{=: \hat x_\nu} &
  &\text{ for all } i\in[1:n],
\end{align*}
allowing us to compute $\hat x_\nu$ for all $\nu\in[1:k]$ in
$\mathcal{O}(nk)$ operations and then evaluate $f_i$ for all
$i\in[1:n]$ in $\mathcal{O}(nk)$ operations, i.e., we have a chance
of reaching \emph{linear} instead of quadratic complexity.

This leaves us with the challenge of finding approximations of
the form \cref{eq:degenerate} that approximate $g$ sufficiently
well.
For gravitational and electrostatic forces, the fast multipole
method \cite{RO85,GRRO87,GRRO97} solves this task by using a special
expansion optimized for this particular function $g$.
For a significantly larger class of functions, standard polynomial
approximations via Taylor expansion \cite{HANO89,SA00} or interpolation
\cite{GI01,BOHA02a} can be employed to similar effect.

For the sake of simplicity, we focus here on interpolation of
$x\mapsto g(y,x)$.
Let $(\xi_{\sigma,\nu})_{\nu=1}^k$ be interpolation points in $\sigma$,
and let $(\ell_{\sigma,\nu})_{\nu=1}^k$ be the corresponding Lagrange
polynomials.
Assuming that the interpolation error is under control, we have
\begin{align*}
  g(y,x)
  &\approx \sum_{\nu=1}^k g(y,\xi_{\sigma,\nu})
                         \ell_{\sigma,\nu}(x) &
  &\text{ for all } x\in\sigma,\ y\in\tau,
\end{align*}
and this is obviously an approximation of the required form
\cref{eq:degenerate}.

Except for very special cases, none of these techniques can give
us a \emph{global} approximation of $g$, i.e., an approximation
that is valid for all $y_i$ and all $x_j$.
This is not surprising since $g$ has a singularity at $x=y$ that
cannot be resolved by an approximation of the form \cref{eq:degenerate}.
Instead we only get approximations on subdomains $\tau\times\sigma$
and have to use multiple subdomains to cover all combinations of
points.
The number of these subdomains can grow very large, frequently
there are $\mathcal{O}(n)$ subdomains, so we need an efficient
approach to handling large numbers of subdomains.

A very successful strategy relies on a hierarchy of subdomains:
assume that every subdomain $\sigma$ is either small, so that the
few points it contains can be treated directly, or subdivided into
two disjoint subdomains $\sigma_1$ and $\sigma_2$.
If the vectors
\begin{align*}
  \hat x_{\sigma_1,\nu}
  &= \sum_{\substack{j=1\\ x_j\in\sigma_1}}^n m_j \ell_{\sigma_1,\nu}(x_j), &
  \hat x_{\sigma_2,\nu}
  &= \sum_{\substack{j=1\\ x_j\in\sigma_2}}^n m_j \ell_{\sigma_2,\nu}(x_j)
\end{align*}
have already been computed, we can re-interpolate the Lagrange
polynomials $\ell_{\sigma,\nu}$ in the points in $\sigma_1$ and
$\sigma_2$, respectively, to obtain
\begin{align*}
  \ell_{\sigma,\nu}
  &\approx \sum_{\nu_1=1}^k \ell_{\sigma,\nu}(\xi_{\sigma_1,\nu_1})
                          \ell_{\sigma_1,\nu_1}, &
  \ell_{\sigma,\nu}
  &\approx \sum_{\nu_2=1}^k \ell_{\sigma,\nu}(\xi_{\sigma_2,\nu_2})
                          \ell_{\sigma_2,\nu_2},
\end{align*}
and therefore
\begin{align*}
  \hat x_{\sigma,\nu}
  &= \sum_{\substack{j=1\\ x_j\in\sigma}}^n m_j \ell_{\sigma,\nu}(x_j)
   = \sum_{\substack{j=1\\ x_j\in\sigma_1}}^n m_j \ell_{\sigma,\nu}(x_j)
   + \sum_{\substack{j=1\\ x_j\in\sigma_2}}^n m_j \ell_{\sigma,\nu}(x_j)\\
  &\approx \sum_{\substack{j=1\\ x_j\in\sigma_1}}^n m_j
     \sum_{\nu_1=1}^k \ell_{\sigma,\nu}(\xi_{\sigma_1,\nu_1})
                    \ell_{\sigma_1,\nu_1}(x_j)
   + \sum_{\substack{j=1\\ x_j\in\sigma_2}}^n m_j
     \sum_{\nu_2=1}^k \ell_{\sigma,\nu}(\xi_{\sigma_2,\nu_2})
                    \ell_{\sigma_2,\nu_2}(x_j)\\
  &= \sum_{\nu_1=1}^k \ell_{\sigma,\nu}(\xi_{\sigma_1,\nu_1})
                    \hat x_{\sigma_1,\nu_1}
   + \sum_{\nu_2=1}^k \ell_{\sigma,\nu}(\xi_{\sigma_2,\nu_2})
                    \hat x_{\sigma_2,\nu_2},
\end{align*}
i.e., we can compute $\hat x_{\sigma,\nu}$ by evaluating only
$2k$ summands.
This approach leads to fast multipole methods \cite{RO85,GRRO87} and
$\mathcal{H}^2$-matrix representations \cite{HAKHSA00,BOHA02,BO10} that
require only a total of $\mathcal{O}(nk)$ operations for \emph{all}
subdomains.

The resulting algorithm interpolates $g$ on a domain
$\tau_0\times\sigma_0$, interpolates the result again on a
smaller domain $\tau_1\times\sigma_1$, which is then interpolated
on an even smaller domain $\tau_2\times\sigma_2$, until very small
domains have been reached.
The subject of this article is to investigate the cumulative
effect of these iterated interpolation steps on the final
error and the stability of the procedure.

Previous results rely either on Taylor expansion \cite{SA00}
or Chebyshev expansions on Bernstein elliptic discs \cite{BOLOME02,BOME15}.
Bernstein discs offer a very precise characterization of the
convergence behaviour of interpolation on intervals \cite{DELO93},
but they have so far only been used in intermediate steps.

The new approach presented in this article is based on a slight
generalization (cf. \cref{th:approximation_error}) of
a well-known result \cite[Theorem~7.8.1]{DELO93} that allows us to
bound the error on an entire Bernstein disc instead of an interval.

Using this error estimate allows us to
\begin{itemize}
  \item obtain a fairly general proof of convergence and stability
        for variable-order methods \cite{SA00,BOLOME02},
        cf. \cref{th:error_first} and \cref{ex:variable_order},
  \item prove that iterated interpolation is stable as long as the
        interpolation orders are not too small,
        cf. \cref{th:stability_first} and \cref{co:stability},
  \item prove that iterated interpolation can be used to 
        approximate derivatives, e.g., to cover the double-layer and
        hypersingular operators in boundary element methods or to
        compute gradients of potentials, cf. \cref{th:derivatives}, and
  \item prove that these results can be generalized to advanced
        approximation techniques for oscillatory kernel functions
        appearing, e.g., in boundary element methods for the
        high-frequency Helmholtz equation \cite{BR91,ENYI07,MESCDA12,BOME15},
        cf. \cref{th:oscillatory}.
\end{itemize}
To keep the presentation simple, this article focuses on the
one-dimensional setting.
Tensor methods can be used to extend the results to multi-dimensional
interpolation.

\Cref{se:bernstein} follows in the footsteps of
\cite[Theorem~7.8.1]{DELO93} to prove the generalized best-approximation
estimate \cref{th:approximation_error} and the interpolation
error estimate \cref{co:interpolation_error}.
\Cref{se:nested} investigates the relationship between
Bernstein discs for nested intervals, with the key result of
\cref{co:nested_discs} showing that if the intervals
shrink uniformly, the transformed Bernstein discs grow uniformly.
\Cref{se:iterated} takes advantage of this property to
prove two error and stability estimates for iterated interpolation:
\cref{th:error_first} is well-suited for variable-order
interpolation, while \cref{th:stability_first} allows us
to handle derivatives of the interpolating polynomial.
\Cref{se:oscillatory} covers a special class of interpolation
operators tailored to the oscillatory kernel function of standard
Helmholtz boundary element methods.

\section{Interpolation on Bernstein discs}
\label{se:bernstein}

Before we discuss iterated interpolation, we briefly recall a few
fundamental results concerning the approximation of holomorphic
functions by interpolation.
A key tool is the \emph{Joukowsky transformation} \cite{JO10}
given by
\begin{align*}
  \gamma\colon \bbbc\setminus\{0\} &\to \bbbc, &
               z &\mapsto \frac{z + 1/z}{2}.
\end{align*}
For every $w\in\bbbc$, we can find a solution $z\in\bbbc$ of
the quadratic equation $z^2 - 2 w z + 1 = 0$, this solution
is non-zero, $1/z$ is also a solution, and we have
$\gamma(z)=\gamma(1/z)=w$, so $\gamma$ is surjective.

The Joukowsky transformation maps the unit circle
$S_1 = \{ z\in\bbbc\ :\ |z|=1 \}$ to the unit interval $[-1,1]$
due to $\gamma(z)=\Re(z)$ for all $z\in S_1$.

The Joukowsky transformation maps the real half-axis $\bbbr_{\geq 1}$
onto itself and is monotonically increasing, i.e., we have
\begin{align}\label{eq:joukowsky_increasing}
  x > y &\iff \gamma(x) > \gamma(y) &
  &\text{ for all } x,y\in\bbbr_{\geq 1}.
\end{align}
Finally, for every $\varrho\in\bbbr_{\geq 1}$, the Joukowsky
transformation maps the (open and closed) annuli
\begin{align*}
  \mathcal{A}_\varrho
  &:= \{ z\in\bbbc\ :\ 1/\varrho < |z| < \varrho \}, &
  \bar{\mathcal{A}}_\varrho
  &:= \{ z\in\bbbc\ :\ 1/\varrho \leq |z| \leq \varrho \}
\end{align*}
to the (open and closed) \emph{Bernstein elliptic discs}
\begin{align*}
  \mathcal{D}_\varrho
  &:= \{ w\in\bbbc\ :\ |w-1| + |w+1| < 2\gamma(\varrho) \},\\
  \bar{\mathcal{D}}_\varrho
  &:= \{ w\in\bbbc\ :\ |w-1| + |w+1| \leq 2\gamma(\varrho) \},
\end{align*}
this follows from \cref{eq:joukowsky_increasing} and
the identity $|\gamma(z)-1|+|\gamma(z)+1| = 2\gamma(|z|)$ for all
$z\in\bbbc\setminus\{0\}$.

This last property allows us to investigate the approximation
of holomorphic functions by polynomials \cite[\S 7.8]{DELO93}:
let $\varrho\in\bbbr_{>1}$, and let $f\colon\mathcal{D}_\varrho\to\bbbc$
be holomorphic.
Then $\hat f := f\circ\gamma$ is holomorphic in the annulus
$\mathcal{A}_\varrho$ and therefore has a Laurent series expansion
\begin{align*}
  \hat f(z) &= \sum_{n=-\infty}^\infty a_n z^n &
  &\text{ for all } z\in\mathcal{A}_\varrho
\end{align*}
with coefficients
\begin{align}
  a_n &:= \frac{1}{2\pi i} \int_{|z|=r} \frac{\hat f(z)}{z^{n+1}} \,dz &
  &\text{ for all } n\in\bbbz,\label{eq:laurent_coeffs}
\end{align}
where any $r\in(1/\varrho,\varrho)$ can be chosen due to
Cauchy's integral theorem.
Since $\gamma(1/z)=\gamma(z)$, we also have $\hat f(1/z)=\hat f(z)$
and obtain $a_{-n}=a_n$ for all $n\in\bbbn$, and therefore
\begin{align*}
  \hat f(z) &= a_0 + 2 \sum_{n=1}^\infty a_n \frac{z^n + z^{-n}}{2} &
  &\text{ for all } z\in\mathcal{A}_\varrho.
\end{align*}
It is easy to verify that the \emph{Chebyshev polynomials} given by
\begin{align*}
  C_n(w) &:= \begin{cases}
    1 &\text{ if } n=0,\\
    w &\text{ if } n=1,\\
    2 w C_{n-1}(w) - C_{n-2}(w) &\text{ otherwise}
  \end{cases} &
  &\text{ for all } w\in\bbbc,\ n\in\bbbn_0
\end{align*}
satisfy the equation
\begin{align*}
  C_n(\gamma(z)) &= \frac{z^n + z^{-n}}{2} &
  &\text{ for all } z\in\bbbc\setminus\{0\},\ n\in\bbbn_0.
\end{align*}
Let $w\in\mathcal{D}_\varrho$.
We have seen that we can find $z\in\mathcal{A}_\varrho$ such that
$w=\gamma(z)$ and therefore
\begin{align*}
  f(w) &= \hat f(z) = a_0 + 2 \sum_{n=1}^\infty a_n \frac{z^n + z^{-n}}{2}
  = a_0 + 2 \sum_{n=1}^\infty a_n C_n(w),
\end{align*}
i.e., the Laurent series of $\hat f$ corresponds to the Chebyshev
expansion of $f$.
Truncating the Chebyshev expansion yields polynomial approximations of $f$.

In order to estimate the approximation error, we require bounds
for the coefficients $a_n$ and the Chebyshev polynomials $C_n$.
We introduce the notation
\begin{equation*}
  \|f\|_{\infty,\Omega}
  := \sup\{ |f(w)|\ :\ w\in\Omega \}
\end{equation*}
for functions $f\colon\Omega\to\bbbc$ and sets $\Omega\subseteq\bbbc$.
For the coefficients \cref{eq:laurent_coeffs} we have
\begin{align}\label{eq:laurent_bound}
  |a_n| &\leq \lim_{r\to\varrho} \frac{\max\{\hat f(z)\ :\ |z|=r\}}{r^n}
         \leq \frac{\|\hat f\|_{\infty,\mathcal{D}_\varrho}}{\varrho^n} &
  &\text{ for all } n\in\bbbn,
\end{align}
while for $\hat\varrho\in[1,\varrho]$ we have
\begin{align}\label{eq:chebyshev_bound}
  |C_n(w)| &\leq \frac{|z|^n + |z|^{-n}}{2} \leq \hat\varrho^n &
  &\text{ for all } w\in\bar{\mathcal{D}}_{\hat\varrho},
\end{align}
where we choose $z\in\bar{\mathcal{A}}_{\hat\varrho}$ with $\gamma(z)=w$.
Combining both estimates yields an error estimate.

%
%
\begin{theorem}[Approximation error]
\label{th:approximation_error}
Let $\varrho\in\bbbr_{>1}$ and $\hat\varrho\in[1,\varrho)$.
Let $f\colon\mathcal{D}_\varrho\to\bbbc$ be holomorphic.
For any $m\in\bbbn$ we can find an $m$-th order polynomial $p$ such
that
\begin{equation*}
  \|f-p\|_{\infty,\bar{\mathcal{D}}_{\hat\varrho}}
  \leq \frac{2}{\varrho/\hat\varrho-1}
       \left(\frac{\hat\varrho}{\varrho}\right)^m
       \|f\|_{\infty,\mathcal{D}_\varrho}.
\end{equation*}
\end{theorem}
\begin{proof}
The proof is a slight modification of \cite[Theorem~7.8.1]{DELO93}.
Let $m\in\bbbn$ and
\begin{equation*}
  p := a_0 + 2 \sum_{n=1}^m a_n C_n.
\end{equation*}
Combining \cref{eq:laurent_bound} and \cref{eq:chebyshev_bound}, we obtain
\begin{align*}
  |f(w)-p(w)| &= 2 \left|\sum_{n=m+1}^\infty a_n C_n(w)\right|
  \leq 2 \sum_{n=m+1}^\infty |a_n|\, |C_n(w)|
  \leq 2 \sum_{n=m+1}^\infty \|f\|_{\infty,\mathcal{D}_\varrho}
                            \varrho^{-n} \hat\varrho^n\\
  &= 2 \|f\|_{\infty,\mathcal{D}_\varrho}
       \sum_{n=m+1}^\infty \left(\frac{\hat\varrho}{\varrho}\right)^n
   = 2 \|f\|_{\infty,\mathcal{D}_\varrho}
       \frac{(\hat\varrho/\varrho)^{m+1}}{1-\hat\varrho/\varrho}
   = 2 \|f\|_{\infty,\mathcal{D}_\varrho}
       \frac{(\hat\varrho/\varrho)^m}{\varrho/\hat\varrho-1}
\end{align*}
by using the geometric series equation.
\end{proof}

Due to $\bar{\mathcal{D}}_1 = [-1,1]$, the special case $\hat\varrho=1$ yields
\begin{equation*}
  \|f-p\|_{\infty,[-1,1]}
  \leq \frac{2}{\varrho-1} \varrho^{-m} \|f\|_{\infty,\mathcal{D}_\varrho}.
\end{equation*}
While proving the existence of an approximating polynomial is
reassuring, practical applications require us to actually
construct such a polynomial.
We will accomplish this task by interpolation:
let $m\in\bbbn$, and let $\xi_0,\ldots,\xi_m\in[-1,1]$ be
pairwise distinct interpolation points and $\ell_0,\ldots,\ell_m$
the corresponding $m$-th order Lagrange polynomials.
We denote the corresponding interpolation operator by
\begin{align}\label{eq:interpolation_operator}
  \mathfrak{I}_m[f] &:= \sum_{\nu=0}^m f(\xi_\nu) \ell_\nu &
  &\text{ for all } f\in C[-1,1].
\end{align}
We have $\mathfrak{I}_m[p] = p$ for any $m$-th order polynomial $p$
and
\begin{align}\label{eq:lebesgue}
  \|\mathfrak{I}_m[f]\|_{\infty,[-1,1]}
  &\leq \Lambda_m \|f\|_{\infty,[-1,1]} &
  &\text{ for all } f\in C[-1,1],
\end{align}
where the \emph{Lebesgue constant} $\Lambda_m$ is given by
\begin{equation*}
  \Lambda_m
  := \max\left\{ \sum_{\nu=0}^m |\ell_\nu(x)|\ :\ x\in[-1,1] \right\}.
\end{equation*}
In order to extend this stability estimate from $[-1,1]$ to a
closed Bernstein disc $\bar{\mathcal{D}}_{\hat\varrho}$, we use the
\emph{Bernstein inequality}.

%
%
\begin{lemma}[Bernstein inequality]
\label{le:bernstein_inequality}
Let $p$ be an $m$-th order polynomial, let $\hat\varrho\in\bbbr_{\geq 1}$.
We have
\begin{equation*}
  \|p\|_{\infty,\bar{\mathcal{D}}_{\hat\varrho}}
  \leq \hat\varrho^m \|p\|_{\infty,[-1,1]}.
\end{equation*}
\end{lemma}
\begin{proof}
cf. \cite[Theorem~4.2.2]{DELO93}
\end{proof}

%
%
\begin{corollary}[Interpolation error]
\label{co:interpolation_error}
Let $\varrho\in\bbbr_{>1}$ and $\hat\varrho\in[1,\varrho)$.
Let $f\colon\mathcal{D}_\varrho\to\bbbc$ be holomorphic.
We have
\begin{equation*}
  \|f-\mathfrak{I}_m[f]\|_{\infty,\bar{\mathcal{D}}_{\hat\varrho}}
  \leq \frac{2 (1+\Lambda_m)}{\varrho/\hat\varrho - 1}
       \left(\frac{\hat\varrho}{\varrho}\right)^m
       \|f\|_{\infty,\mathcal{D}_\varrho}.
\end{equation*}
\end{corollary}
\begin{proof}
Let $p$ be the $m$-th order polynomial constructed in
\cref{th:approximation_error}.
Due to $\mathfrak{I}_m[p]=p$ and \cref{le:bernstein_inequality},
we have
\begin{align*}
  \|f-\mathfrak{I}_m[f]\|_{\infty,\bar{\mathcal{D}}_{\hat\varrho}}
  &= \|f-p+\mathfrak{I}_m[p-f]\|_{\infty,\bar{\mathcal{D}}_{\hat\varrho}}
   \leq \|f-p\|_{\infty,\bar{\mathcal{D}}_{\hat\varrho}}
      + \|\mathfrak{I}_m[f-p]\|_{\infty,\bar{\mathcal{D}}_{\hat\varrho}}\\
  &\leq \|f-p\|_{\infty,\bar{\mathcal{D}}_{\hat\varrho}}
      + \hat\varrho^m \|\mathfrak{I}_m[f-p]\|_{\infty,[-1,1]}\\
  &\leq \|f-p\|_{\infty,\bar{\mathcal{D}}_{\hat\varrho}}
      + \hat\varrho^m \Lambda_m \|f-p\|_{\infty,[-1,1]}\\
  &\leq \frac{2}{\varrho/\hat\varrho-1}
      \left(\frac{\hat\varrho}{\varrho}\right)^m \|f\|_{\infty,\mathcal{D}_\varrho} 
      + \hat \varrho^m \Lambda_m \frac{2}{\varrho-1} 
      \left(\frac{1}{\varrho}\right)^m \|f\|_{\infty,\mathcal{D}_\varrho}\\
  &\leq \frac{2 (1+\Lambda_m)}{\varrho/\hat\varrho - 1}
      \left(\frac{\hat\varrho}{\varrho}\right)^m
      \|f\|_{\infty,\mathcal{D}_\varrho}
\end{align*}
due to $1\leq \hat\varrho<\varrho$ and therefore
$\varrho \geq \varrho/\hat\varrho > 1$.
\end{proof}

For our investigation, we need interpolation operators of different
order on general domains.
Let $(\mathfrak{I}_m)_{m=1}^\infty$ be a family of interpolation
operators of the type \cref{eq:interpolation_operator} on the
reference interval $[-1,1]$ with corresponding Lebesgue numbers
$(\Lambda_m)_{m=1}^\infty$.
For an interval $[a,b]$, $a<b$, we use the simple transformation
\begin{align*}
  \Phi_{a,b}\colon [-1,1] &\to [a,b], &
                  x &\mapsto \frac{b+a}{2} + \frac{b-a}{2} x,
\end{align*}
to define the transformed interpolation operators
$\mathfrak{I}_{[a,b],m}$ for all $m\in\bbbn$ by
\begin{align*}
  \mathfrak{I}_{[a,b],m}[f]
  &:= \mathfrak{I}_m[f\circ\Phi_{a,b}] \circ \Phi_{a,b}^{-1} &
  &\text{ for all } f\in C[a,b].
\end{align*}
For our error estimates, we introduce the transformed
Bernstein elliptic discs
\begin{align*}
  \mathcal{D}_{[a,b],\varrho} &:= \Phi_{a,b}(\mathcal{D}_\varrho), &
  \bar{\mathcal{D}}_{[a,b],\varrho} &:= \Phi_{a,b}(\bar{\mathcal{D}}_\varrho)
\end{align*}
and the short notation
\begin{align*}
  \|f\|_{[a,b],\varrho}
  &:= \|f\|_{\infty,\bar{\mathcal{D}}_{[a,b],\varrho}}
   = \max\{ |f(w)|\ :\ w\in\bar{\mathcal{D}}_{[a,b],\varrho} \} &
  &\text{ for all } f\in C(\bar{\mathcal{D}}_{[a,b],\varrho}).
\end{align*}
In a slight abuse of notation, we apply this norm also to functions
with domains larger than $\bar{\mathcal{D}}_{[a,b],\varrho}$.
\Cref{co:interpolation_error} takes the following form:

%
%
\begin{corollary}[Interpolation error]
\label{co:interpolation_transformed}
Let $\varrho\in\bbbr_{>1}$ and $\hat\varrho\in[1,\varrho)$.
Let $a,b\in\bbbr$ with $a<b$.
Let $f\colon\bar{\mathcal{D}}_{[a,b],\varrho}\to\bbbc$ be holomorphic.
We have
\begin{equation*}
  \|f-\mathfrak{I}_{[a,b],m}[f]\|_{[a,b],\hat\varrho}
  \leq \frac{2 (1+\Lambda_m)}{\varrho/\hat\varrho-1}
       \left(\frac{\hat\varrho}{\varrho}\right)^m
       \|f\|_{[a,b],\varrho}.
\end{equation*}
\end{corollary}
\begin{proof}
The function $\hat f := f\circ\Phi_{a,b}$ is holomorphic in
$\mathcal{D}_\varrho$.
\Cref{co:interpolation_error} yields
\begin{align*}
  \|f-\mathfrak{I}_{[a,b],m}[f]\|_{[a,b],\hat\varrho}
  &= \|f\circ\Phi_{a,b}
      - \mathfrak{I}_m[f\circ\Phi_{a,b}]\|_{\infty,\bar{\mathcal{D}}_{\hat\varrho}}
   = \|\hat f
        - \mathfrak{I}_m[\hat f]\|_{\infty,\bar{\mathcal{D}}_{\hat\varrho}}\\
  &\leq \frac{2 (1+\Lambda_m)}{\varrho/\hat\varrho-1}
        \left(\frac{\hat\varrho}{\varrho}\right)^m
        \|\hat f\|_{\infty,\mathcal{D}_\varrho}
   = \frac{2 (1+\Lambda_m)}{\varrho/\hat\varrho-1}
        \left(\frac{\hat\varrho}{\varrho}\right)^m
        \|f\|_{[a,b],\varrho},
\end{align*}
where we have used $\mathfrak{I}_{[a,b],m}[f]\circ\Phi_{a,b}
= \mathfrak{I}_m[f\circ\Phi_{a,b}]$.
\end{proof}

\section{Bernstein discs for nested intervals}
\label{se:nested}

Since \cref{co:interpolation_transformed} requires $\hat\varrho<\varrho$,
we can expect iterated interpolation to work only if the Bernstein disc
$\mathcal{D}_{[a,b],\hat\varrho}$ for an interval $[a,b]\subsetneq[-1,1]$
is contained in the Bernstein disc $\mathcal{D}_{[-1,1],\varrho}$ with
$\hat\varrho>\varrho$.
If we can ensure that the ratio between the lengths of $[a,b]$
and $[-1,1]$ is bounded and that $\hat\varrho$ is not
too small, we can prove $\hat\varrho\geq\sigma\varrho$ with $\sigma>1$
and thus obtain an estimate for the rate of convergence.

%
%
\begin{lemma}[Nested Bernstein discs]
\label{le:nested_bernstein_discs}
Let $a,b\in\bbbr$ with $-1\leq a<b\leq 1$, let $\delta := \frac{b-a}{2}$.
The function
\begin{align*}
  \gamma^\dag \colon \bbbr_{\geq 1} &\to \bbbr_{\geq 1}, &
        \varrho &\mapsto \varrho + \sqrt{\varrho^2-1},
\end{align*}
satisfies $\gamma(\gamma^\dag(\varrho))=\varrho$ for all
$\varrho\in\bbbr_{\geq 1}$.
For $\varrho\in\bbbr_{>1}$ let
\begin{equation*}
  \varrho_{a,b}
  := \gamma^\dag\left(\frac{\gamma(\varrho)-1}{\delta}+1\right).
\end{equation*}
We have $\mathcal{D}_{[a,b],\varrho_{a,b}}\subseteq\mathcal{D}_\varrho$.
\end{lemma}
\begin{proof}
Let $w\in\mathcal{D}_{[a,b],\varrho_{a,b}}$.
By definition, this means that there is a
$\hat w\in\mathcal{D}_{\varrho_{a,b}}$ with $w=\Phi_{a,b}(\hat w)$.
We observe
\begin{align*}
  |w-b| + |w-a|
  &= \left|\frac{b+a}{2} + \frac{b-a}{2} \hat w - b\right|
   + \left|\frac{b+a}{2} + \frac{b-a}{2} \hat w - a\right|\\
  &= \left|\delta \hat w - \frac{b-a}{2}\right|
   + \left|\delta \hat w + \frac{b-a}{2}\right|
   = \delta |\hat w-1| + \delta |\hat w+1|
   < 2 \delta \gamma(\varrho_{a,b}).
\end{align*}
For all $\varrho\in\bbbr_{\geq 1}$, we have
\begin{equation*}
  \gamma(\gamma^\dag(\varrho))
  = \frac{1}{2} \left(
        \varrho+\sqrt{\varrho^2-1}
        + \frac{1}{\varrho+\sqrt{\varrho^2-1}} \right)
  = \frac{(\varrho+\sqrt{\varrho^2-1})^2 + 1}
         {2 (\varrho + \sqrt{\varrho^2-1})}
  = \varrho.
\end{equation*}
Due to $b\leq 1$ and $-1\leq a$ and using the
definition of $\varrho_{a,b}$, we have
\begin{align*}
  |w-1| + |w+1|
  &= |w-b+b-1| + |w-a+a+1|\\
  &\leq |w-b| + |b-1| + |w-a| + |a+1|\\
  &= |w-b| + 1-b + |w-a| + a+1\\
  &< 2 \delta \gamma(\varrho_{a,b}) + 2 - (b-a)
   = 2 \left( \delta \gamma(\varrho_{a,b}) + 1 - \delta \right)\\
  &= 2 \left( \delta \left(\frac{\gamma(\varrho)-1}{\delta}+1\right)
               + 1 - \delta \right)
   = 2 \gamma(\varrho),
\end{align*}
and therefore $w\in\mathcal{D}_\varrho$.
\end{proof}

If we want to interpolate a holomorphic function $f$ given in
$\mathcal{D}_\varrho$ on the subinterval $[a,b]\subseteq[-1,1]$, we find
that the function is holomorphic in $\mathcal{D}_{[a,b],\varrho_{a,b}}$
and that the error in $\mathcal{D}_{[a,b],\varrho}$ will converge at
a rate of $\varrho/\varrho_{a,b}$.
We have
\begin{equation*}
  \gamma(\varrho_{a,b})
  = \frac{\gamma(\varrho)-1}{\delta} + 1
  = \gamma(\varrho) + \frac{1-\delta}{\delta} (\gamma(\varrho)-1)
  > \gamma(\varrho),
\end{equation*}
and \cref{eq:joukowsky_increasing} yields $\varrho_{a,b}>\varrho$,
i.e., we can expect exponential convergence.
Finding a bound for the rate of convergence is a slightly more
challenging task.

%
%
\begin{lemma}[Rate of convergence]
\label{le:rate_of_convergence}
The function
\begin{align*}
  \hat\sigma \colon \bbbr_{\geq 1}\times(0,1) &\to \bbbr_{\geq 1}, &
     (\varrho,\delta) &\mapsto
     \frac{\gamma^\dag\left(\frac{\gamma(\varrho)-1}{\delta}+1\right)}
          {\varrho},
\end{align*}
is monotonically increasing in $\varrho$ and monotonically
decreasing in $\delta$ with the limits
\begin{align*}
  \hat\sigma(1,\delta) &= 1, &
  \lim_{\varrho\to\infty} \hat\sigma(\varrho,\delta) &= 1/\delta &
  &\text{ for all } \delta\in(0,1).
\end{align*}
\end{lemma}
\begin{proof}
Since $\gamma^\dag$ is monotonically increasing, $\hat\sigma$ is
monotonically decreasing in $\delta$.

To prove that $\hat\sigma$ is monotonically increasing in $\varrho$,
we fix $\delta\in(0,1)$.
Due to \cref{eq:joukowsky_increasing}, the identity
$\gamma(\gamma^\dag(\gamma(\varrho)))=\gamma(\varrho)$ implies
$\gamma^\dag(\gamma(\varrho))=\varrho$ for all $\varrho\in\bbbr_{\geq 1}$,
and we have
\begin{align*}
  \hat\sigma(\varrho,\delta)
  &= \frac{\gamma^\dag\left(\frac{\gamma(\varrho)-1}{\delta}+1\right)}
          {\gamma^\dag(\gamma(\varrho))} &
  &\text{ for all } \varrho\in\bbbr_{\geq 1}.
\end{align*}
We already know that $\gamma$ is monotonically increasing, so it
suffices to prove that
\begin{align*}
  g \colon \bbbr_{\geq 1} &\to \bbbr, &
           x &\mapsto
            \frac{\gamma^\dag\left(\frac{x-1}{\delta}+1\right)}
                 {\gamma^\dag(x)},
\end{align*}
is monotonically increasing.
Using
\begin{align*}
  \frac{\partial}{\partial\varrho} \gamma^\dag(\varrho)
  &= 1 + \frac{2 \varrho}{2 \sqrt{\varrho^2-1}}
   = \frac{\gamma^\dag(\varrho)}{\sqrt{\varrho^2-1}} &
  &\text{ for all } \varrho\in\bbbr_{\geq 1},
\end{align*}
the chain and quotient rules yield
\begin{align*}
  \frac{\partial g}{\partial x}(x)
  &= \frac{\frac{\gamma^\dag\left(\frac{x-1}{\delta}+1\right)}
                {\delta\sqrt{\left(\frac{x-1}{\delta}+1\right)^2-1}}
           \gamma^\dag(x)
           - \gamma^\dag\left(\frac{x-1}{\delta}+1\right)
             \frac{\gamma^\dag(x)}{\sqrt{x^2-1}}}
          {(\gamma^\dag(x))^2}\\
  &= \frac{1 - \delta \frac{\sqrt{\left(\frac{x-1}{\delta}+1\right)^2-1}}
                           {\sqrt{x^2-1}}}
          {\gamma^\dag(x)
           \delta \sqrt{\left(\frac{x-1}{\delta}+1\right)^2-1}}
     \gamma^\dag\left(\frac{x-1}{\delta}+1\right).
\end{align*}
Since $\gamma^\dag$ maps into $\bbbr_{\geq 1}$, it suffices to prove
\begin{equation*}
  0 \leq 1 - \delta \frac{\sqrt{\left(\frac{x-1}{\delta}+1\right)^2-1}}
                      {\sqrt{x^2-1}}
    = 1 - \sqrt{\frac{(x-1+\delta)^2-\delta^2}
                     {x^2-1}}.
\end{equation*}
Due to $\delta<1$, we have
\begin{equation*}
  \frac{(x-1+\delta)^2-\delta^2}{x^2-1}
  = \frac{(x-1)^2 + 2 \delta(x-1)}{(x+1)(x-1)}
  = \frac{x-1 + 2 \delta}{x+1}
  < 1,
\end{equation*}
i.e., $g$ is monotonically increasing, and so is
$\varrho \mapsto \hat\sigma(\varrho,\delta) = g(\gamma(\varrho))$.

The identity $\hat\sigma(1,\delta)=1$ follows directly from $\gamma(1)=1$
and $\gamma^\dag(1)=1$.
Due to
\begin{align*}
  \lim_{x\to\infty} \frac{\gamma^\dag(x)}{x} &= 2, &
  \lim_{\varrho\to\infty} \frac{\frac{\gamma(\varrho)-1}{\delta}+1}
                       {\varrho} &= \frac{1}{2\delta},
\end{align*}
we obtain
\begin{equation*}
  \lim_{\varrho\to\infty} \hat\sigma(\varrho,\delta)
  = \lim_{\varrho\to\infty}
  \frac{\gamma^\dag\left(\frac{\gamma(\varrho)-1}{\delta}+1\right)}
       {\frac{\gamma(\varrho)-1}{\delta}+1}
  \frac{\frac{\gamma(\varrho)-1}{\delta}+1}{\varrho}
  = \frac{2}{2\delta} = \frac{1}{\delta},
\end{equation*}
where we have used that $\varrho\mapsto\frac{\gamma(\varrho)-1}{\delta}+1$
grows to infinity as $\varrho\to\infty$.
\end{proof}

Combining \cref{le:nested_bernstein_discs} with
\cref{le:rate_of_convergence} allows us to estimate the size
of Bernstein discs around an interval $[a,b]$ contained in a larger
interval $[c,d]$.

%
%
\begin{corollary}[Nested discs]
\label{co:nested_discs}
Let $\varrho_0\in\bbbr_{>1}$ and $\delta_0\in(0,1)$.
There is a $\sigma\in\bbbr_{>1}$ such that
\begin{align*}
  \mathcal{D}_{[a,b],\sigma\varrho}
  &\subseteq \mathcal{D}_{[c,d],\varrho} &
  &\text{ for all } \varrho\geq\varrho_0,\ 
      c\leq a<b \leq d \text{ with }
      b-a \leq \delta_0 (d-c).
\end{align*}
\end{corollary}
\begin{proof}
Using the function $\hat\sigma$ introduced in
\cref{le:rate_of_convergence}, we choose
$\sigma := \hat\sigma(\varrho_0,\delta_0)$.

Let $\varrho\in\bbbr_{\geq\varrho_0}$ and $a,b,c,d\in\bbbr$ with
$c\leq a<b\leq d$, and $(b-a)\leq\delta_0(d-c)$.
In order to apply \cref{le:nested_bernstein_discs}, we
have to transform $[c,d]$ to the reference interval $[-1,1]$.
Due to $a,b\in[c,d]$, we have
\begin{align*}
  \hat a &:= \frac{2}{d-c} \left(a - \frac{d+c}{2}\right) \in [-1,1], &
  \hat b &:= \frac{2}{d-c} \left(b - \frac{d+c}{2}\right) \in [-1,1]
\end{align*}
and $\Phi_{c,d}(\hat a)=a$ as well as $\Phi_{c,d}(\hat b)=b$.
We let
\begin{equation*}
  \delta := \frac{\hat b-\hat a}{2}
  = \frac{b-a}{d-c} \leq \delta_0
\end{equation*}
and use \cref{le:rate_of_convergence} to find
$\sigma = \hat\sigma(\varrho_0,\delta_0) \leq \hat\sigma(\varrho,\delta)$,
so that \cref{le:nested_bernstein_discs} yields
$\mathcal{D}_{[\hat a,\hat b],\sigma\varrho} \subseteq \mathcal{D}_\varrho$.
Applying $\Phi_{c,d}$ gives us
\begin{equation*}
  \mathcal{D}_{[a,b],\sigma\varrho}
  = \Phi_{c,d}(\mathcal{D}_{[\hat a,\hat b],\sigma\varrho})
  \subseteq \Phi_{c,d}(\mathcal{D}_\varrho)
  = \mathcal{D}_{[c,d],\varrho},
\end{equation*}
and the proof is complete.
\end{proof}

\section{Iterated interpolation}
\label{se:iterated}

Approximation schemes like variable-order $\mathcal{H}^2$-ma\-tri\-ces
\cite{SA00,BOLOME02,BOSA03} and $\mathcal{DH}^2$-matrices for the
high-frequency Helmholtz equation \cite{ENYI07,BOME15} rely on interpolation
along a nested sequence
\begin{equation*}
  [a_L,b_L] \subseteq [a_{L-1},b_{L-1}] \subseteq \ldots
  \subseteq [a_1,b_1] \subseteq [a_0,b_0]
\end{equation*}
of intervals: we first interpolate a given function $f$ on the second-largest
interval $[a_1,b_1]$, then interpolate the result again on the third-largest
interval $[a_2,b_2]$, and repeat the process until we reach $[a_L,b_L]$.
Our task is to prove that this sequence of interpolation steps leads to
a reasonable approximation of the original function $f$.

In order to investigate the interpolation error for different
orders, we require the family of interpolation operators to
be \emph{stable}, i.e., we assume that there are constants
$\Lambda,\lambda\in\bbbr_{>0}$ such that
\begin{align}\label{eq:interpolation_stable}
  \Lambda_m &\leq \Lambda (1+m)^\lambda &
  &\text{ for all } m\in\bbbn.
\end{align}
Chebyshev interpolation satisfies this assumption with
$\Lambda=\lambda=1$ \cite{RI90}.
Using this assumption, we obtain a more convenient estimate
for the interpolation error.

%
%
\begin{theorem}[Interpolation error]
\label{th:interpolation_error}
Let $\sigma\in\bbbr_{>1}$ and $q\in(1/\sigma,1]$.
There is a constant $\Cin$ depending only on
  \cref{eq:interpolation_stable}, $\sigma$ and $q$
such that for all $\varrho\in\bbbr_{\geq 1}$, all $\tau\in\bbbr_{\geq 1}$,
all $a,b\in\bbbr$ with $a<b$ and all holomorphic
$f\colon\mathcal{D}_{[a,b],\sigma\tau\varrho}\to\bbbr$ we have
\begin{align*}
  \|f-\mathfrak{I}_{[a,b],m}[f]\|_{[a,b],\varrho}
  &\leq \Cin q^m \tau^{-m}
        \|f\|_{[a,b],\sigma\tau\varrho} &
  &\text{ for all } m\in\bbbn.
\end{align*}
\end{theorem}
\begin{proof}
Due to the stability condition \cref{eq:interpolation_stable} the supremum
\begin{equation}\label{eq:Cin}
  \Cin := \sup\left\{ \frac{2 (1+\Lambda_m)}{(\sigma-1)}
                      \left(\frac{1}{\sigma q}\right)^m
                      \ :\ m\in\bbbn \right\}
\end{equation}
is finite, since $\sigma q > 1$ implies that the exponential
$(\sigma q)^m$ grows faster than $\Lambda_m$ as $m$ increases.

Now let $\varrho\in\bbbr_{\geq\varrho_0}$, $\tau\in\bbbr_{\geq 1}$,
$a,b\in\bbbr$ with $a<b$ and $m\in\bbbn$.
Let $f$ be a function that is holomorphic in
$\mathcal{D}_{[a,b],\sigma\tau\varrho}$.
We apply \cref{co:interpolation_transformed} and obtain
\begin{align*}
  \|f &- \mathfrak{I}_{[a,b],m}[f]\|_{[a,b],\varrho}
   \leq \frac{2 (1+\Lambda_m)}{\sigma\tau-1}
         \left(\frac{\varrho}
                    {\sigma\tau\varrho}\right)^m
         \|f\|_{\mathcal{D}_{[a,b],\sigma\tau\varrho}}\\
  &\leq \frac{2 (1+\Lambda_m)}{\sigma-1}
         \left(\frac{1}{\sigma q}\right)^m
         q^m \tau^{-m}
         \|f\|_{\mathcal{D}_{[a,b],\sigma\tau\varrho}}
   \leq \Cin q^m \tau^{-m}
         \|f\|_{\mathcal{D}_{[a,b],\sigma\tau\varrho}}.
\end{align*}
\end{proof}

In order to be able to apply \cref{th:interpolation_error} to sequences
of intervals, we assume that there is a $\delta_0\in(0,1)$ such that
\begin{align}\label{eq:shrinking_intervals}
  b_\ell-a_\ell &\leq \delta_0 (b_{\ell-1}-a_{\ell-1}) &
  &\text{ for all } \ell\in[1:L].
\end{align}
Our goal is to analyze the iterated interpolation operators given by
\begin{align*}
  \mathfrak{I}_{j,i}
  &:= \begin{cases}
        I & \text{ if } i=j,\\
        \mathfrak{I}_{[a_j,b_j],m_j} \circ \cdots \circ
          \mathfrak{I}_{[a_{i+1},b_{i+1}],m_{i+1}} & \text{ otherwise}
      \end{cases}
  &\text{ for all } i,j\in[0:L],\ i\leq j,
\end{align*}
where $m_1,\ldots,m_L\in\bbbn$ are the orders of interpolation.

For the investigation of the stability and the error of nested
interpolation, we can follow two different approaches:
the ``approximation first'' approach relies on the telescoping sum
\begin{equation}\label{eq:error_first_sum}
  f - \mathfrak{I}_{j,i}[f]
  = \sum_{\ell=i+1}^j \mathfrak{I}_{j,\ell}[ f -
    \mathfrak{I}_{[a_\ell,b_\ell],m_\ell}[f] ].
\end{equation}
Treating the telescoping sum with the triangle inequality
means that we need error estimates for
$f-\mathfrak{I}_{[a_\ell,b_\ell],m_\ell}[f]$
and stability estimates for $\mathfrak{I}_{j,\ell}$.
For the error estimates, we can take advantage of
\cref{co:nested_discs} in combination with
\cref{th:interpolation_error} for $\tau=\sigma^{\ell-1}$ to
obtain error estimates of the form
\begin{align*}
  \|f-\mathfrak{I}_{[a_\ell,b_\ell],m_\ell}[f]\|_{[a_\ell,b_\ell],\varrho}
  &\leq \Cin q^{m_\ell} \sigma^{-m_\ell(\ell-1)}
        \|f\|_{[a_\ell,b_\ell],\sigma^\ell \varrho}
   \leq \Cin q^{\ell m_\ell} \|f\|_{[a_0,b_0],\varrho},
\end{align*}
i.e., the rate of convergence increases with $\ell$, clearly a
very desirable property.
On the other hand, obtaining stability estimates independent of $L$
for $\|\mathfrak{I}_{j,\ell}[f]\|_{[a_j,b_j],\varrho}$ poses a challenge
unless variable-order techniques are employed.

A second approach relies -- somewhat counter-intuitively --- on
error estimates not for $f$ itself, but for its interpolating
polynomial $\mathfrak{I}_{\ell-1,i}[f]$.
This ``stability first'' approach relies on the telescoping sum
\begin{equation}\label{eq:stability_first_sum}
  f - \mathfrak{I}_{j,i}[f]
  = \sum_{\ell=i+1}^j \mathfrak{I}_{\ell-1,i}[f]
        - \mathfrak{I}_{[a_\ell,b_\ell],m_\ell}[\mathfrak{I}_{\ell-1,i}[f]].
\end{equation}
Replacing $\sigma$ provided by \cref{co:nested_discs} by
$\sigma^\theta$ for $\theta\in(0,1]$ in \cref{th:interpolation_error}
with $\sigma^{1-\theta} \varrho$ instead of $\varrho$,
$\sigma^\theta$ instead of $\sigma$ and $\tau=1$ gives us estimates of the form
\begin{align*}
  \|f-\mathfrak{I}_{[a_\ell,b_\ell],m_\ell}[f]\|_{[a_\ell,b_\ell],\sigma^{1-\theta}\varrho}
  &\leq \Cin q^{\theta m_\ell} \|f\|_{[a_\ell,b_\ell],\sigma\varrho}
   \leq \Cin q^{\theta m_\ell} \|f\|_{[a_{\ell-1},b_{\ell-1}],\varrho},
\end{align*}
i.e., we sacrifice convergence speed to gain error estimates on larger
Bernstein discs.

%
%
\begin{theorem}[``approximation first'' approach]
\label{th:error_first}
Let $\varrho_0\in\bbbr_{>1}$, let $\sigma\in\bbbr_{>1}$ be as in
\cref{co:nested_discs}, let $q\in(1/\sigma,1]$ and $\Cin$ be chosen
as in \cref{th:interpolation_error}.
Let $\varrho\in\bbbr_{\geq\varrho_0}$ and let
$f\colon\mathcal{D}_{[a_0,b_0],\varrho}\to\bbbc$ be holomorphic.
We have
\begin{subequations}
\begin{gather}
  \|\mathfrak{I}_{j,i}[f]\|_{[a_j,b_j],\varrho}
   \leq \left(\prod_{\ell=i+1}^j (1 + \Cin q^{m_\ell})\right)
            \|f\|_{[a_i,b_i],\varrho},
  \label{eq:error_first_stability}\\
  \|f - \mathfrak{I}_{j,i}[f]\|_{[a_j,b_j],\varrho}
   \leq \sum_{k=i+1}^j \left(\prod_{\ell=k+1}^j (1 + \Cin q^{m_\ell})\right)
            \Cin q^{m_k (k-r)} \|f\|_{[a_r,b_r],\varrho}
  \label{eq:error_first_accuracy}
\end{gather}
\end{subequations}
for all $i,j,r\in[0:L]$ with $r\leq i\leq j$.
\end{theorem}
\begin{proof}
We first prove the stability estimate \cref{eq:error_first_stability}.
The triangle inequality, \cref{th:interpolation_error} with
$\tau=1$, and \cref{co:nested_discs} give us
\begin{align*}
  \|\mathfrak{I}_{[a_\ell,b_\ell],m_\ell} &[f]\|_{[a_\ell,b_\ell],\varrho}
   \leq \|f\|_{[a_\ell,b_\ell],\varrho}
      + \|f -
      \mathfrak{I}_{[a_\ell,b_\ell],m_\ell}[f]\|_{[a_\ell,b_\ell],\varrho}\\
  &\leq \|f\|_{[a_\ell,b_\ell],\sigma\varrho}
      + \Cin q^{m_\ell} \|f\|_{[a_\ell,b_\ell],\sigma\varrho}
   \leq (1 + \Cin q^{m_\ell}) \|f\|_{[a_{\ell-1},b_{\ell-1}],\varrho}
\end{align*}
for all $\ell\in[1:L]$.
A simple induction yields \cref{eq:error_first_stability}.

We prove the error estimate \cref{eq:error_first_accuracy} by
induction over $j-i\in\bbbn_0$.
The case $j=i$ is trivial.

Let now $n\in\bbbn_0$ be such that \cref{eq:error_first_accuracy}
holds for all $i,j\in[0:L]$ with $j-i=n$.

Let $i,j,r\in[0:L]$ with $j-i=n+1$ and $r\leq i$.
Using the triangle inequality, the stability estimate
\cref{eq:error_first_stability}, \cref{th:interpolation_error}
with $\tau=\sigma^{i-r}$, and \cref{co:nested_discs} (applied $i+1-r$ times),
we obtain
\begin{align*}
  \|f - &\mathfrak{I}_{j,i}[f]\|_{[a_j,b_j],\varrho}
   \leq \|f - \mathfrak{I}_{j,i+1}[f]\|_{[a_j,b_j],\varrho}
      + \|\mathfrak{I}_{j,i+1}[f -
            \mathfrak{I}_{[a_{i+1},b_{i+1}],m_{i+1}}[f]]\|_{[a_j,b_j],\varrho}\\
  &\leq \|f - \mathfrak{I}_{j,i+1}[f]\|_{[a_j,b_j],\varrho}\\
  &\quad + \left(\prod_{\ell=i+2}^j (1+\Cin q^{m_\ell})\right)
        \|f - \mathfrak{I}_{[a_{i+1},b_{i+1}],m_{i+1}}[f]
        \|_{[a_{i+1},b_{i+1}],\varrho}\\
  &\leq \|f - \mathfrak{I}_{j,i+1}[f]\|_{[a_j,b_j],\varrho}\\
  &\quad + \left(\prod_{\ell=i+2}^j (1+\Cin q^{m_\ell})\right)
        \Cin q^{m_{i+1}} \sigma^{-(i-r) m_{i+1}}
        \|f\|_{[a_{i+1},b_{i+1}],\sigma^{i+1-r}\varrho}\\
  &\leq \sum_{k=i+2}^j \left(\prod_{\ell=k+1}^j (1+\Cin q^{m_\ell})\right)
          \Cin q^{m_k(k-r)} \|f\|_{[a_r,b_r],\varrho}\\
  &\quad + \left(\prod_{\ell=i+2}^j (1+\Cin q^{m_\ell})\right)
        \Cin q^{m_{i+1} (i+1-r)} \|f\|_{[a_r,b_r],\varrho}\\
  &= \sum_{k=i+1}^j \left(\prod_{\ell=k+1}^j (1+\Cin q^{m_\ell})\right)
          \Cin q^{m_k (k-r)} \|f\|_{[a_r,b_r],\varrho},
\end{align*}
relying on the induction assumption in the last step.
\end{proof}

%
%
\begin{example}[Variable-order interpolation]
\label{ex:variable_order}
When using $\mathcal{H}^2$-matrix methods to approximate certain
integral operators, variable-order interpolation schemes
\cite{SA00,BOLOME02,BOSA03} can be very efficient:
in order to reduce the storage requirements, we choose the orders as
$m_\ell = \alpha + \beta (L-\ell)$ with
$\alpha,\beta\in\bbbn$, i.e., we use large orders on large intervals
and small orders on small intervals.

If we choose $q<1$ in \cref{th:error_first}, we have
\begin{align*}
  \prod_{\ell=i+1}^j (1 + \Cin q^{m_\ell})
  &\leq \prod_{\ell=i+1}^j \exp(\Cin q^{m_\ell})
   = \exp\left(\Cin \sum_{\ell=i+1}^j q^{\alpha+\beta(L-\ell)}\right)\\
  &= \exp\left(\Cin q^{\alpha+\beta(L-j)}
                  \sum_{\ell=i+1}^j q^{\beta(j-\ell)}\right)\\
  &\leq \exp\left(\Cin \frac{q^{\alpha+\beta(L-j)}}{1-q^\beta}\right)
   \leq \exp\left(\Cin \frac{q^\alpha}{1-q^\beta}\right) =: \Cst,
\end{align*}
by the geometric sum formula, i.e., the variable-order interpolation
is \emph{uniformly stable} for all $L$, $i,j\in[0,L]$ with $i\leq j$.

In order to obtain an error estimate, we have to investigate
the terms $q^{m_k (k-r)}$ for $k\in[i+1,j]$.
For the sake of simplicity, we consider only the case $i=r=0$,
$j=L>1$, and can use
\begin{align*}
  m_k k
  &= (\alpha + \beta (L-k)) k
   = \alpha k + \beta (L-k) + \beta (L-k) (k-1)\\
  &\geq \min\{\alpha,\beta\} L
      + \begin{cases}
          \beta \lceil L/2 \rceil (k-1) &\text{ if } k\leq \lfloor L/2 \rfloor,\\
          \beta (L-k) \lfloor L/2 \rfloor &\text{ otherwise},
        \end{cases}
\end{align*}
to find
\begin{align*}
  \|f - &\mathfrak{I}_{L,0}[f]\|_{[a_L,b_L],\varrho}
   \leq \sum_{k=1}^L \Cst \Cin q^{m_k k} \|f\|_{[a_0,b_0],\varrho}\\
  &\leq \Cst \Cin q^{\min\{\alpha,\beta\} L} \left(
          \sum_{k=1}^{\lfloor L/2\rfloor} q^{\beta \lceil L/2 \rceil (k-1)} +
          \sum_{k=\lceil L/2\rceil}^L q^{\beta \lfloor L/2 \rfloor (L-k)}
          \right) \|f\|_{[a_0,b_0],\varrho}\\
  &\leq 2 \Cst \Cin q^{\min\{\alpha,\beta\} L}
          \sum_{k=0}^\infty q^{\beta \lfloor L/2 \rfloor k} \|f\|_{[a_0,b_0],\varrho}
   \leq 2 \Cst \Cin \frac{q^{\min\{\alpha,\beta\} L}}
                        {1 - q^{\beta \lfloor L/2 \rfloor}} \|f\|_{[a_0,b_0],\varrho}
\end{align*}
by the geometric summation formula.
The term $q^{\min\{\alpha,\beta\} L}$ lets the accuracy grow exponentially
as $L$ increases without the need to adjust the parameters $\alpha$ and $\beta$.
\end{example}

%
%
\begin{remark}[Shrinking condition]
The condition \cref{eq:shrinking_intervals} can be weakened:
if we have $[a_2,b_2]\subseteq[a_1,b_1]\subseteq[a_0,b_0]$, but
only $(b_1-a_1) \leq \delta_0 (b_0-a_0)$, we can apply
\cref{th:interpolation_error} to $\sigma^{1/2}$
instead of $\sigma$ and obtain
\begin{align*}
  \|f - \mathfrak{I}_{[a_2,b_2],m_2}[f]\|_{[a_2,b_2],\varrho}
  &\leq \Cin q^{m_2/2} \|f\|_{[a_2,b_2],\sigma^{1/2}\varrho}
   \leq \Cin q^{m_2/2} \|f\|_{[a_1,b_1],\sigma^{1/2}\varrho},\\
  \|f - \mathfrak{I}_{[a_1,b_1],m_1}[f]\|_{[a_1,b_1],\sigma^{1/2}\varrho}
  &\leq \Cin q^{m_1/2} \|f\|_{[a_1,b_1],\sigma\varrho}
   \leq \Cin q^{m_1/2} \|f\|_{[a_0,b_0],\varrho},
\end{align*}
i.e., the convergence rate is worse, but the basic structure of
the stability and convergence proofs can be preserved.
\end{remark}

%
%
\begin{theorem}[``stability first'' approach]
\label{th:stability_first}
Let $\varrho_0\in\bbbr_{>1}$, let $\theta_1,\theta_2\in(0,1)$
with $\theta_1+\theta_2=1$, let $\sigma\in\bbbr_{>1}$ be as in
\cref{co:nested_discs}, and let $q_1\in(\sigma^{-\theta_1},1)$
and $q_2:=\sigma^{-\theta_2}$.
There is a $\Cin\in\bbbr_{>0}$ such that for all
$\varrho\in\bbbr_{\geq\varrho_0}$ and all holomorphic
$f\colon\mathcal{D}_{[a_0,b_0],\varrho}\to\bbbc$ we have
\begin{subequations}
\begin{gather}
  \|\mathfrak{I}_{j,i}[f]\|_{[a_j,b_j],\sigma^{\theta_2(j-i)}\varrho}
  \leq \prod_{\ell=i+1}^j (1 + \Cin q_1^{m_\ell})
           \|f\|_{[a_i,b_i],\varrho},
  \label{eq:stability_first_stability}\\
  \|f - \mathfrak{I}_{j,i}[f]\|_{[a_j,b_j],\varrho}
  \leq \Cin \sum_{k=i+1}^j q_2^{m_k (k-i)} q_1^{m_k}
         \left(\prod_{\ell=i+1}^{k-1} (1 + \Cin q_1^{m_\ell})\right)
         \|f\|_{[a_i,b_i],\varrho}
  \label{eq:stability_first_accuracy}
\end{gather}
for all $i,j\in[0:L]$ with $j\geq i$.
\end{subequations}
\end{theorem}
\begin{proof}
We apply \cref{th:interpolation_error} to $\sigma^{\theta_1}$,
$q_1$, and $\sigma^{\theta_2(\ell-i)} \varrho$ in place of $\sigma$, $q$,
and $\varrho$ to get a constant $\Cin$ depending only
on $\sigma$ and $q_1$.
Using the triangle inequality and $\sigma^{\theta_2(\ell-i)}\varrho
\geq\varrho$, we obtain
\begin{align*}
  \|\mathfrak{I}_{[a_\ell,b_\ell],m_\ell}&[f]
  \|_{[a_\ell,b_\ell],\sigma^{\theta_2 (\ell-i)} \varrho}
   \leq \|f\|_{[a_\ell,b_\ell],\sigma^{\theta_2 (\ell-i)} \varrho}
      + \|f - \mathfrak{I}_{[a_\ell,b_\ell],m_\ell}[f]
        \|_{[a_\ell,b_\ell],\sigma^{\theta_2 (\ell-i)} \varrho}\\
  &\leq (1 + \Cin q_1^{m_\ell})
        \|f\|_{[a_\ell,b_\ell],\sigma^{\theta_2 (\ell-i) + \theta_1}\varrho}
   = (1 + \Cin q_1^{m_\ell})
        \|f\|_{[a_\ell,b_\ell],\sigma^{\theta_2 (\ell-i-1) + 1}\varrho}\\
  &\leq (1 + \Cin q_1^{m_\ell})
   \|f\|_{[a_{\ell-1},b_{\ell-1}],\sigma^{\theta_2 (\ell-i-1)}\varrho}
\end{align*}
for all $\ell\in[i+1:L]$, where we use \cref{co:nested_discs}
in the last step.
A simple induction leads to \cref{eq:stability_first_stability}.

We will prove \cref{eq:stability_first_accuracy} again by induction
over $j-i\in\bbbn_0$.
The case $j=i$ is trivial.

Let now $n\in\bbbn_0$ be such that \cref{eq:stability_first_accuracy}
holds for all $i,j\in[0:L]$ with $j-i=n$.

Let $i,j\in[0:L]$ with $j-i=n+1$.
We have
\begin{align*}
  \|f - &\mathfrak{I}_{j,i}[f]\|_{[a_j,b_j],\varrho}
   \leq \|f - \mathfrak{I}_{j-1,i}[f]\|_{[a_j,b_j],\varrho}\\
  &\quad + \|\mathfrak{I}_{j-1,i}[f]
          - \mathfrak{I}_{[a_j,b_j],m_j}[\mathfrak{I}_{j-1,i}[f]]
        \|_{[a_j,b_j],\varrho}.
\end{align*}
The first term can be handled by the induction assumption due to
$j-1-i=n$.
For the second term, we use \cref{th:interpolation_error}
with $\tau = \sigma^{-\theta_2(j-i)}$ to get
\begin{align}
  \|\mathfrak{I}_{j-1,i}[f]
        &- \mathfrak{I}_{[a_j,b_j],m_j}[\mathfrak{I}_{j-1,i}[f]]
           \|_{[a_j,b_j],\varrho}\notag\\
  &\leq \Cin q_1^{m_j} \sigma^{-\theta_2(j-i)m_j} \|\mathfrak{I}_{j-1,i}[f]
                    \|_{[a_j,b_j],\sigma^{\theta_2(j-i) +\theta_1}
                      \varrho}\notag\\
  &= \Cin q_2^{(j-i) m_j} q_1^{m_j}
       \|\mathfrak{I}_{j-1,i}[f]
       \|_{[a_j,b_j],\sigma^{\theta_2(j-i-1)+1}\varrho}\notag\\
  &\leq \Cin q_2^{(j-i) m_j} q_1^{m_j}
       \|\mathfrak{I}_{j-1,i}[f]
       \|_{[a_{j-1},b_{j-1}],\sigma^{\theta_2(j-i-1)}\varrho}\notag\\
  &\leq \Cin q_2^{(j-i) m_j} q_1^{m_j}
           \left(\prod_{\ell=i+1}^{j-1} (1 + \Cin q_1^{m_\ell})\right)
           \|f\|_{[a_i,b_i],\varrho},\label{eq:stability_first_single}
\end{align}
where we have used the stability estimate \cref{eq:stability_first_stability}
in the last step.
Combining this estimate with the induction assumption yields
\begin{equation*}
  \|f - \mathfrak{I}_{j,i}[f]\|_{[a_j,b_j],\varrho}
  \leq \Cin \sum_{k=i+1}^j
          q_2^{m_k (k-i)} q_1^{m_k}
          \left( \prod_{\ell=i+1}^{k-1} (1+\Cin q_1^{m_\ell})\right)
          \|f\|_{[a_i,b_i],\varrho},
\end{equation*}
completing the induction.
\end{proof}

%
%
\begin{corollary}[Stability]
\label{co:stability}
Let $\varrho_0\in\bbbr_{>1}$, let $\sigma\in\bbbr_{>1}$
be as in \cref{co:nested_discs}, let $\theta_1,\theta_2\in(0,1)$ with
$\theta_1+\theta_2=1$, let $q_1\in(\sigma^{-\theta_1},1)$ and
$q_2=\sigma^{-\theta_2}$, and let $\Cin$ be as in \cref{th:stability_first}.

There are $\alpha_0\in\bbbn$ and $\Cst\in\bbbr_{\geq 1}$, $\Cax\in\bbbr_{>0}$
such that if $\alpha := \min\{ m_\ell\ :\ \ell\in[1:L] \} \geq \alpha_0$
holds, we have
\begin{align*}
  \|\mathfrak{I}_{j,i}[f]\|_{[a_j,b_j],\varrho}
  &\leq \Cst \|f\|_{[a_i,b_i],\varrho},\\
  \|f-\mathfrak{I}_{j,i}[f]\|_{[a_j,b_j],\varrho}
  &\leq \Cax q_1^\alpha q_2^\alpha \|f\|_{[a_i,b_i],\varrho}
\end{align*}
for all $i,j\in[0:L]$ with $i\leq j$, all $\varrho\in\bbbr_{\geq\varrho_0}$
and all holomorphic $f\colon\mathcal{D}_{[a_0,b_0],\varrho}\to\bbbc$.
\end{corollary}
\begin{proof}
Let $i,j\in[0:L]$ with $i\leq j$.
Let $\alpha:=\min\{ m_\ell\ :\ \ell\in[1:L] \}$.
\cref{eq:stability_first_accuracy} yields
\begin{align*}
  \|f - \mathfrak{I}_{j,i}[f]\|_{[a_j,b_j],\varrho}
  &\leq \sum_{k=i+1}^j
     \Cin q_2^{\alpha(k-i)} q_1^\alpha
     \left(\prod_{\ell=i+1}^{k-1} (1 + \Cin q_1^\alpha) \right)
     \|f\|_{[a_i,b_i],\varrho}\\
  &\leq \Cin q_1^\alpha q_2^\alpha \sum_{k=i+1}^j
     \bigl( (1 + \Cin q_1^\alpha) q_2^\alpha \bigr)^{k-i-1}
     \|f\|_{[a_i,b_i],\varrho}.
\end{align*}
Now we choose $\alpha_0\in\bbbn$ large enough to guarantee
$(1+\Cin q_1^{\alpha_0}) q_2^{\alpha_0} \leq 1/2$, assume
$\alpha\geq\alpha_0$, and use the geometric sum equation to conclude
\begin{equation*}
  \|f - \mathfrak{I}_{j,i}[f]\|_{[a_j,b_j],\varrho}
  \leq 2 \Cin q_1^\alpha q_2^\alpha \|f\|_{[a_i,b_i],\varrho}.
\end{equation*}
Choosing $\Cax := 2 \Cin$ proves the error estimate, and the
triangle inequality yields
\begin{equation*}
  \|\mathfrak{I}_{j,i}[f]\|_{[a_j,b_j],\varrho}
  \leq \|f\|_{[a_j,b_j],\varrho}
       + \|f - \mathfrak{I}_{j,i}[f]\|_{[a_j,b_j],\varrho}
  \leq (1 + \Cax q_1^\alpha q_2^\alpha) \|f\|_{[a_i,b_i],\varrho},
\end{equation*}
so we get the stability estimate with $\Cst := 1 + \Cax q_1^\alpha q_2^\alpha$.
\end{proof}

The ``stability first'' approach can be used to obtain error estimates
for the derivatives of the interpolation error, allowing us to approximate
the derivatives of a function by the derivatives of its interpolating
polynomial.
The key tool is Cauchy's bound for the derivatives of holomorphic functions.

%
%
\begin{lemma}[Cauchy's inequality]
\label{le:cauchy_inequality}
Let $\varrho_0\in\bbbr_{>1}$.
There is a constant $\Cca\in\bbbr_{>0}$ such that
\begin{equation}\label{eq:derivative}
  \|f'\|_{\infty,[a,b]}
  \leq \frac{\Cca}{b-a} \|f\|_{[a,b],\varrho}
\end{equation}
holds for all $\varrho\in\bbbr_{\geq\varrho_0}$, $a,b\in\bbbr$ with $a<b$
and all functions $f$ holomorphic in $\mathcal{D}_{[a,b],\varrho}$.
\end{lemma}
\begin{proof}
Let $\varrho\in\bbbr_{\geq\varrho_0}$ and
$r:=\frac{(\varrho-1)^2}{2 \varrho}$.
A straightforward computation reveals
\begin{align*}
  \{ w\in\bbbc\ :\ |w-x|<r \}
  &\subseteq \mathcal{D}_\varrho &
  &\text{ for all } x\in[-1,1].
\end{align*}
Let $f$ be holomorphic in $\mathcal{D}_\varrho$.
Using Cauchy's inequality for derivatives, we find
\begin{equation*}
  |f'(x)| \leq \max\left\{ \frac{|f(w)|}{\tilde r}\ :\ w\in\bbbc,
        \ |w-x|=\tilde r \right\}
  \leq \frac{\|f\|_{[-1,1],\varrho}}{\tilde r}
\end{equation*}
for all $x\in[-1,1]$ and $\tilde r\in(0,r)$, and therefore
\begin{equation*}
  \|f'\|_{\infty,[-1,1]}
  \leq \frac{2 \varrho}{(\varrho-1)^2} \|f\|_{[-1,1],\varrho}
  \leq \frac{\Cca}{2} \|f\|_{[-1,1],\varrho}
\end{equation*}
with $\Cca := \tfrac{4 \varrho_0}{(\varrho_0-1)^2}
\geq \tfrac{4 \varrho}{(\varrho-1)^2}$.
A straightforward scaling argument using $\Phi_{a,b}' = \tfrac{b-a}{2}$
completes the proof.
\end{proof}

In order to keep the denominator $b-a$ in the estimate
\cref{eq:derivative} under control, we assume that there is
a $\delta_1\in\bbbr_{>0}$ with
\begin{align}\label{eq:shrinking_slowly}
  b_\ell-a_\ell &\geq \delta_1 (b_{\ell-1}-a_{\ell-1}) &
  &\text{ for all } \ell\in[1:L],
\end{align}
this is a counterpart of the ``shrinking assumption''
\cref{eq:shrinking_intervals}.

%
%
\begin{theorem}[Derivatives]
\label{th:derivatives}
Let $\varrho_0\in\bbbr_{>1}$, let $\theta_1,\theta_2\in(0,1)$ with
$\theta_1+\theta_2=1$, let $\sigma\in\bbbr_{>1}$
be as in \cref{co:nested_discs}, let $q_1\in(\sigma^{-\theta_1},1)$ and
$q_2=\sigma^{-\theta_2}$, and let $\Cin$ be as in
\cref{th:stability_first} and $\Cca$ as
\cref{le:cauchy_inequality}.

There are $\alpha_0\in\bbbn$ and $\Cax\in\bbbr_{>0}$ such that
if $\alpha := \min\{m_\ell\ :\ \ell\in[1:L] \} \geq \alpha_0$ holds,
we have
\begin{align*}
  \|(f-\mathfrak{I}_{j,i}[f])'\|_{[a_j,b_j],\varrho}
  &\leq \frac{\Cax}{b_i-a_i} q_2^\alpha \|f\|_{[a_i,b_i],\varrho}
\end{align*}
for all $i,j\in[0:L]$ with $i\leq j$, all $\varrho\in\bbbr_{\geq\varrho_0}$,
and all holomorphic $f\colon\mathcal{D}_{[a_0,b_0],\varrho}\to\bbbc$.
\end{theorem}
\begin{proof}
We modify the proof of \cref{th:stability_first}:
by \cref{le:cauchy_inequality} and \cref{eq:stability_first_single},
we obtain
\begin{align*}
  \|(\mathfrak{I}_{j-1,i}[f]
      &- \mathfrak{I}_{[a_j,b_j],m_j}[\mathfrak{I}_{j-1,i}[f]])'
           \|_{\infty,[a_j,b_j]}\\
  &\leq \frac{\Cca}{b_j-a_j}
     \|\mathfrak{I}_{j-1,i}[f]
         - \mathfrak{I}_{[a_j,b_j],m_j}[\mathfrak{I}_{j-1,i}[f]]
     \|_{[a_j,b_j],\varrho}\\
  &\leq \frac{\Cca}{b_j-a_j}
     \Cin q_2^{m_j(j-i)} q_1^{m_j}
     \left(\prod_{\ell=i+1}^{j-1} (1 + \Cin q_1^{m_\ell})\right)
     \|f\|_{[a_i,b_i],\varrho}\\
  &\leq \frac{\Cca}{b_i-a_i}
     \Cin \left(\frac{q_2^{m_j}}{\delta_1}\right)^{j-i} q_1^{m_j}
     \left(\prod_{\ell=i+1}^{j-1} (1 + \Cin q_1^{m_\ell})\right)
     \|f\|_{[a_i,b_i],\varrho}
\end{align*}
for all $i,j\in[0:L]$ with $i\leq j-1$.
Using the same strategy as in \cref{co:stability}, we choose
$\alpha_0\in\bbbn$ large enough to ensure
\begin{equation*}
  \frac{q_2^{\alpha_0}}{\delta_1} (1+\Cin q_1^{\alpha_0}) \leq \frac{1}{2}.
\end{equation*}
Assuming $\alpha := \min\{ m_\ell\ :\ \ell\in[1:L] \} \geq \alpha_0$
yields
\begin{align*}
  \|(\mathfrak{I}_{j-1,i}[f]
      &- \mathfrak{I}_{[a_j,b_j],m_j}[\mathfrak{I}_{j-1,i}[f]])'
           \|_{\infty,[a_j,b_j]}
   \leq \frac{\Cca}{b_i-a_i} \Cin q_1^\alpha
      \left(\frac{1}{2}\right)^{j-i} \|f\|_{[a_i,b_i],\varrho}.
\end{align*}
Using induction as in \cref{th:stability_first} gives us
\begin{align*}
  \|(f - \mathfrak{I}_{j,i}[f])'\|_{\infty,[a_j,b_j]}
  &\leq \sum_{k=i+1}^j \frac{\Cca}{b_i-a_i} \Cin q_1^\alpha
         \left(\frac{1}{2}\right)^{k-i} \|f\|_{[a_i,b_i],\varrho}\\
  &\leq \frac{2 \Cca}{b_i-a_i} \Cin q_1^\alpha
         \|f\|_{[a_i,b_i],\varrho}.
\end{align*}
Setting $\Cax := 2 \Cca \Cin$ completes the proof.
\end{proof}

\section{Iterated interpolation of oscillatory functions}
\label{se:oscillatory}

The kernel function
\begin{equation*}
  g(x,y) = \frac{\exp(\iota \kappa \|x-y\|)}{4\pi \|x-y\|}
\end{equation*}
of the three-dimensional Helmholtz operator oscillates quickly
if the wave number $\kappa$ is large.
This means that standard interpolation is a poor fit for
constructing fast methods for Helmholtz boundary element methods.

An effective solution is to split the kernel function
into a plane wave and a locally smooth remainder that can be
approximated \cite{BR91,ENYI07,MESCDA12,BOME15}:
we choose a unit vector $\hat c\in\bbbr^3$ and apply interpolation
to the modified kernel function
\begin{equation*}
  g_c(x,y) = \frac{\exp(\iota \kappa (\|x-y\| - \langle \hat c,x-y \rangle))}
                  {4\pi \|x-y\|},
\end{equation*}
then we can use
\begin{equation*}
  g(x,y) = \exp(\iota \kappa \langle \hat c,x-y \rangle)\, g_c(x,y)
         = \exp(\iota \kappa \langle \hat c,x \rangle)\,
           \overline{\exp(\iota \kappa \langle \hat c,y \rangle)}\,
           g_c(x,y)
\end{equation*}
to reconstruct the original kernel function.
Multiplication by
\begin{equation*}
  \exp(\iota \kappa \langle \hat c,x \rangle)
  = \exp(\iota \kappa \hat c_1 x_1)\,
    \exp(\iota \kappa \hat c_2 x_2)\,
    \exp(\iota \kappa \hat c_3 x_3)
\end{equation*}
is a tensor operation, therefore we can restrict our analysis to
the one-dimensional multiplication operators $\mathfrak{E}_c$ given by
\begin{equation*}
  \mathfrak{E}_c[f](x) := \exp(\iota c x)\, f(x),
\end{equation*}
where $c$ is the product of $\hat c_i$ and the wave number $\kappa$.
Instead of interpolating $g$ directly, we divide by a plane wave,
i.e., apply $\mathfrak{E}_{-c}$, interpolate the result, and then
multiply by the plane wave again, i.e., apply $\mathfrak{E}_c$.
Our task is to investigate the resulting ``oscillatory interpolation
operators''
\begin{equation*}
  \mathfrak{I}_{[a,b],m,c}
  := \mathfrak{E}_c \circ \mathfrak{I}_{[a,b],m} \circ \mathfrak{E}_{-c}.
\end{equation*}
In order to obtain efficient numerical schemes, we have to use
\emph{iterated} oscillatory interpolation, i.e., we again fix
a sequence
\begin{equation*}
  [a_L,b_L] \subseteq [a_{L-1},b_{L-1}] \subseteq \ldots
  \subseteq [a_1,b_1] \subseteq [a_0,b_0]
\end{equation*}
of nested intervals with corresponding ``directions''
$c_0,\ldots,c_L\in\bbbr$ and interpolation orders
$m_1,\ldots,m_L\in\bbbn$.
The iterated interpolation operators are now given by
\begin{align*}
  \mathfrak{I}_{j,i}
  &:= \begin{cases}
    I &\text{ if } i=j,\\
    \mathfrak{I}_{[a_j,b_j],m_j,c_j}\circ\mathfrak{I}_{j-1,i}
    &\text{ otherwise}
  \end{cases} &
  &\text{ for all } i,j\in[0:L],\ i\leq j.
\end{align*}
We again need the shrinking condition \cref{eq:shrinking_intervals},
and we require the ``directions'' of neighbouring steps in the
interpolation chain to be sufficiently close, i.e., we assume that
there is a constant $\omega\in\bbbr_{>0}$ such that
\begin{align*}
  |c_\ell-c_{\ell-1}|\, (b_\ell-a_\ell) &\leq \omega &
  &\text{ for all } \ell\in[1:L].
\end{align*}

%
%
\begin{lemma}[Bounded oscillations]
\label{le:bounded_oscillations}
Let $\varrho\in\bbbr_{>1}$.
There is a constant $\Cos\in\bbbr_{>0}$ such that
for all $i,j\in[0:L]$ with $i\leq j$ and all holomorphic
$f\colon\mathcal{D}_{[a_i,b_i],\varrho}\to\bbbc$ we have
\begin{align*}
  \|\mathfrak{E}_{c_j-c_i}[f]\|_{[a_j,b_j],\varrho}
  = \|\mathfrak{E}_{c_i-c_j}[f]\|_{[a_j,b_j],\varrho}
  &\leq \Cos \|f\|_{[a_j,b_j],\varrho}.
\end{align*}
\end{lemma}
\begin{proof}
We first prove
\begin{align}\label{eq:longrange_directions}
  |c_j - c_i|\, (b_j-a_j) &\leq \frac{\omega}{1-\delta_0} &
  &\text{ for all } i,j\in[0:L],\ i\leq j
\end{align}
with $\delta_0$ from \cref{eq:shrinking_intervals}
by induction over $j-i\in\bbbn_0$.
The case $j=i$ is trivial.

Let now $n\in\bbbn_0$ be such that \cref{eq:longrange_directions}
holds for all $i,j\in[0:L]$ with $j-i=n$.

Let $i,j\in[0:L]$ with $j-i=n+1$.
The triangle equality, \cref{eq:shrinking_intervals}, and the
induction assumption yield
\begin{align*}
  |c_j - c_i|\, (b_j-a_j)
  &\leq |c_j - c_{j-1}|\, (b_j-a_j) + |c_{j-1} - c_i|\, (b_j-a_j)\\
  &\leq \omega + \delta_0 |c_{j-1} - c_i|\,(b_{j-1}-a_{j-1})
   \leq \omega + \delta_0 \frac{\omega}{1-\delta_0}\\
  &\leq \omega \frac{(1-\delta_0) + \delta_0}{1-\delta_0}
   = \frac{\omega}{1-\delta_0}.
\end{align*}
Let now $i,j\in[0:L]$ with $i\leq j$, let
$w\in\mathcal{D}_{[a_j,b_j],\varrho}$, and let
$\hat w\in\mathcal{D}_\varrho$ with $\Phi_{a,b}(\hat w)=w$.
We can find $\hat z=x+\iota y\in\mathcal{A}_\varrho$ with
$\gamma(\hat z)=\hat w$ and get
\begin{equation*}
  \hat w = \frac{1}{2}\left( x+\iota y + \frac{x-\iota y}{x^2+y^2} \right)
    = \frac{x}{2}\left(1+\frac{1}{x^2+y^2}\right)
      + \iota \frac{y}{2}\left(1-\frac{1}{x^2+y^2}\right),
\end{equation*}
which allows us to conclude $|\Im(\hat w)| \leq \tfrac{\varrho}{2}
(1 - \tfrac{1}{\varrho^2}) = \tfrac{\varrho-1/\varrho}{2}$ and
$|\Im(w)| \leq (b_j-a_j) \tfrac{\varrho-1/\varrho}{4}$.
We obtain
\begin{align*}
  |\exp(\iota (c_j-c_i) w)|
  &\leq \exp(|c_j-c_i|\, |\Im(w)|)
   \leq \exp\left(|c_j-c_i|\, (b_j-a_j)\, \tfrac{\varrho-1/\varrho}{4}\right)\\
  &\leq \exp\left(\tfrac{\omega}{1-\delta_0}
                  \tfrac{\varrho-1/\varrho}{4}\right) =: \Cos,
\end{align*}
and this proves our claim, since $\mathfrak{E}_{c_j-c_i}$ is a simple
multiplication operator.
\end{proof}

In this section's setting, $f$ may be oscillatory, i.e., it may grow
exponentially along the imaginary axis.
Therefore we cannot expect the absolute value of the error to converge
reasonably well in a Bernstein disc $\mathcal{D}_{[a,b],\varrho}$.
We can, however, investigate the ``smoothed'' error obtained via
the operator $\mathfrak{E}_{-c}$ eliminating the exponential
growth.

%
%
\begin{theorem}[Oscillatory interpolation]
\label{th:oscillatory}
Let $\varrho=\varrho_0\in\bbbr_{>1}$, let $\sigma\in\bbbr_{>1}$ be as in
\cref{co:nested_discs}, let $q\in(1/\sigma,1]$ and
$\Cin$ be chosen as in \cref{th:interpolation_error}.
Let $\Cos$ be the constant of \cref{le:bounded_oscillations}
with $\sigma\varrho$ instead of $\varrho$.
Let $f\colon\mathcal{D}_{[a_0,b_0],\varrho}\to\bbbc$ be holomorphic.
We have
\begin{subequations}
\begin{align}
  \|\mathfrak{E}_{-c_i} \mathfrak{I}_{j,i}[f]
  \|_{[a_j,b_j],\varrho}
  &\leq \prod_{\ell=i+1}^j (1 + \Cos^2 \Cin q^{m_\ell})
  \|\mathfrak{E}_{-c_i} f\|_{[a_i,b_i],\varrho},
  \label{eq:oscillatory_stability}\\
  \|\mathfrak{E}_{-c_i} (f - \mathfrak{I}_{j,i}[f])
  \|_{[a_j,b_j],\varrho}
  &\leq \sum_{k=i+1}^j \left(\prod_{\ell=i+1}^{k-1}
        (1 + \Cos^2 \Cin q^{m_\ell}) \right)
  \label{eq:oscillatory_accuracy}\\
  &\qquad\qquad \Cos^2 \Cin
        q^{m_k} \|\mathfrak{E}_{-c_i} f\|_{[a_i,b_i],\varrho}
  \notag
\end{align}
\end{subequations}
for all $i,j\in[0:L]$ with $i\leq j$.
\end{theorem}
\begin{proof}
Let $i\in[0:L]$.
Using \cref{th:interpolation_error}, \cref{le:bounded_oscillations}, and
\cref{co:nested_discs}, we find
\begin{align}
  \|\mathfrak{E}_{-c_i}(f - &\mathfrak{I}_{[a_\ell,b_\ell],m_\ell,c_\ell}[f])
  \|_{[a_\ell,b_\ell],\varrho}
   = \|\mathfrak{E}_{c_\ell-c_i}
       ( \mathfrak{E}_{-c_\ell} f
         - \mathfrak{I}_{[a_\ell,b_\ell],m_\ell}
           [\mathfrak{E}_{-c_\ell} f] )\|_{[a_\ell,b_\ell],\varrho}\notag\\
  &\leq \Cos \|\mathfrak{E}_{-c_\ell} f
         - \mathfrak{I}_{[a_\ell,b_\ell],m_\ell}
           [\mathfrak{E}_{-c_\ell} f]\|_{[a_\ell,b_\ell],\varrho}\notag
   \leq \Cos \Cin q^{m_\ell}
            \|\mathfrak{E}_{-c_\ell} f\|_{[a_\ell,b_\ell],\sigma\varrho}\notag\\
  &\leq \Cos \Cin q^{m_\ell}
            \|\mathfrak{E}_{c_i-c_\ell} \mathfrak{E}_{-c_i} f
            \|_{[a_\ell,b_\ell],\sigma\varrho}\notag
   \leq \Cos^2 \Cin q^{m_\ell}
            \|\mathfrak{E}_{-c_i} f\|_{[a_\ell,b_\ell],\sigma\varrho}\notag\\
  &\leq \Cos^2 \Cin q^{m_\ell}
            \|\mathfrak{E}_{-c_i} f\|_{[a_{\ell-1},b_{\ell-1}],\varrho}
  \label{eq:oscillatory_step}
\end{align}
for all $\ell\in[i+1:L]$.
The triangle inequality gives us
\begin{align*}
  \|\mathfrak{E}_{-c_i} \mathfrak{I}_{[a_\ell,b_\ell],m_\ell,c_\ell}[f]
  \|_{[a_\ell,b_\ell],\varrho}
  &\leq \|\mathfrak{E}_{-c_i} f\|_{[a_\ell,b_\ell],\varrho}\\
  &\qquad + \|\mathfrak{E}_{-c_i} (f -
        \mathfrak{I}_{[a_\ell,b_\ell],m_\ell,c_\ell}[f])\|_{[a_\ell,b_\ell],\varrho}\\
  &\leq (1 + \Cos^2 \Cin q^{m_\ell})
        \|\mathfrak{E}_{-c_i} f\|_{[a_{\ell-1},b_{\ell-1}],\varrho}
\end{align*}
for all $\ell\in[i+1:L]$, and a straightforward induction
yields \cref{eq:oscillatory_stability}.

We prove \cref{eq:oscillatory_accuracy} by induction over
$j-i\in\bbbn_0$.
The case $j=i$ is trivial.

Let now $n\in\bbbn_0$ be such that \cref{eq:oscillatory_accuracy}
holds for all $i,j\in[0:L]$ with $j-i=n$.

Let $i,j\in[0:L]$ with $j-i=n+1$.
The triangle inequality, the error estimate \cref{eq:oscillatory_step},
and the stability estimate \cref{eq:oscillatory_stability} yield
\begin{align*}
  \|\mathfrak{E}_{-c_i} (f - &\mathfrak{I}_{j,i}[f])\|_{[a_j,b_j],\varrho}
   \leq \|\mathfrak{E}_{-c_i} (f - \mathfrak{I}_{j-1,i}[f])
          \|_{[a_j,b_j],\varrho}\\
  &\qquad + \|\mathfrak{E}_{-c_i} (
           \mathfrak{I}_{j-1,i}[f]
           - \mathfrak{I}_{[a_j,b_j],m_j,c_j}[
               \mathfrak{I}_{j-1,i}[f]]\|_{[a_j,b_j],\varrho}\\
  &\leq \|\mathfrak{E}_{-c_i} (f - \mathfrak{I}_{j-1,i}[f])
          \|_{[a_j,b_j],\varrho}\\
  &\qquad + \Cin \Cos^2 q^{m_j}
          \|\mathfrak{E}_{-c_i} \mathfrak{I}_{j-1,i}[f]
          \|_{[a_{j-1},b_{j-1}],\varrho}\\
  &\leq \|\mathfrak{E}_{-c_i} (f - \mathfrak{I}_{j-1,i}[f])
          \|_{[a_j,b_j],\varrho}\\
  &\qquad + \Cin \Cos^2 q^{m_j}
          \left(\prod_{\ell=i+1}^{j-1} (1 + \Cin \Cos^2 q^{m_\ell})\right)
          \|\mathfrak{E}_{-c_i} f\|_{[a_i,b_i],\varrho}\\
  &= \sum_{k=i+1}^j \Cin \Cos^2 q^{m_k} \left(\prod_{\ell=i+1}^{k-1}
          (1 + \Cin \Cos^2 q^{m_\ell}) \right)
          \|\mathfrak{E}_{-c_i} f\|_{[a_i,b_i],\varrho},
\end{align*}
where we have used the induction assumption in the last step.
\end{proof}

%
%
\begin{remark}[Stability condition]
\label{rm:stability_condition}
Unfortunately, we cannot use the approach of
\cref{th:stability_first} to obtain uniform stability estimates
for oscillatory interpolation, since \cref{le:bounded_oscillations} only
holds on \emph{fixed} Bernstein discs.
To ensure stability, we need an additional assumption
(cf. \cite[Theorem~5.6]{BOME15}).
Let $\alpha := \min\{ m_\ell\ :\ \ell\in[1:L] \}$ denote the
minimal order of interpolation.
In order to have a stable method, we have to keep
\begin{align*}
  \prod_{\ell=i+1}^j (1 + \Cos^2 \Cin q^{m_\ell})
  &\leq \prod_{\ell=i+1}^j (1 + \Cos^2 \Cin q^\alpha)
   \leq \prod_{\ell=i+1}^j \exp(\Cos^2 \Cin q^\alpha)\\
  &= \exp(\Cos^2 \Cin (j-i) q^\alpha)
\end{align*}
under control for all $i,j\in[0:L]$ with $i\leq j$.
To do so, we choose $p\in(q,1]$ and require
\begin{equation*}
  \alpha \geq \frac{\log(L)}{\log(p)-\log(q)}
  \iff \log(L) + \log\left(\frac{q}{p}\right) \alpha \leq 0
  \iff L \left(\frac{q}{p}\right)^\alpha \leq 1.
\end{equation*}
Due to $q/p\geq q$, this implies $L q^\alpha \leq 1$, and we obtain
\begin{align}
  \|\mathfrak{E}_{-c_i} \mathfrak{I}_{j,i}[f]\|_{[a_j,b_j],\varrho}
  &\leq \exp(\Cos^2 \Cin) \|\mathfrak{E}_{-c_i} f\|_{[a_i,b_i],\varrho},
        \label{eq:oscillatory_stable}\\
  \|\mathfrak{E}_{-c_i} (f - \mathfrak{I}_{j,i}[f])\|_{[a_j,b_j],\varrho}
  &\leq \Cin \Cos^2 \exp(\Cos^2 \Cin) L q^\alpha
        \|\mathfrak{E}_{-c_i} f\|_{[a_i,b_i],\varrho}\notag\\
  &\leq \Cin \Cos^2 \exp(\Cos^2 \Cin) p^\alpha
        \|\mathfrak{E}_{-c_i} f\|_{[a_i,b_i],\varrho}\notag
\end{align}
for all $i,j\in[0:L]$, $i\leq j$, i.e., the iterated oscillatory
interpolation is stable and converges at almost the same rate as
standard interpolation.
\end{remark}

%
%
\begin{remark}[Error estimates]
Due to $|\exp(\iota c_i w)| = 1$ for all $w\in\bbbr$, we have
$\|f - \mathfrak{I}_{j,i}[f]\|_{\infty,[a_j,b_j]}
\leq \|\mathfrak{E}_{-c_i} (f - \mathfrak{I}_{j,i}[f])\|_{[a_j,b_j],\varrho}$
for all $\varrho\in\bbbr_{\geq 1}$.

Standard discretization schemes with a mesh width of $h$ usually
satisfy $\kappa h\lesssim 1$.
This translates to $|c_L| (b_L-a_L) \lesssim 1$, i.e.,
to a bound for $\mathfrak{E}_{-c_L}$, and we find
\begin{align*}
  \|f - &\mathfrak{I}_{L,i}[f]\|_{[a_L,b_L],\varrho}
   \lesssim \|\mathfrak{E}_{-c_L} (f - \mathfrak{I}_{L,i}[f])
            \|_{[a_L,b_L],\varrho}\\
  &\leq \Cos \|\mathfrak{E}_{-c_i} (f - \mathfrak{I}_{L,i}[f])
            \|_{[a_L,b_L],\varrho}.
\end{align*}
Using this estimate in combination with \cref{th:oscillatory}, we can
even apply \cref{le:cauchy_inequality} to obtain estimates for the
derivative of the error.
\end{remark}

In order to use our approximation and stability results in
higher-dimensional settings, it is frequently useful to have
stability estimates that only require $f$ to be bounded on
the interval $[a_0,b_0]=\mathcal{D}_{[a_0,b_0],1}$ instead of on a
Bernstein disc $\mathcal{D}_{[a_0,b_0],\varrho}$ with $\varrho>1$.
Using the ``stability first'' approach, we can obtain estimates
of this type, at least for constant-order interpolation.

%
%
\begin{corollary}[Stability]
\label{co:oscillatory_stability}
Let $m_\ell=\alpha$ for all $\ell\in[1:L]$.
There are $\alpha_0\in\bbbn$ and $\Cst\in\bbbr_{\geq 1}$ such
that if $\alpha\geq\alpha_0$ holds, we have
\begin{align}
  \|\mathfrak{I}_{j,i}[f]\|_{\infty,[a_j,b_j]}
  &\leq \Cst \Lambda_{m_{i+1}} \|f\|_{\infty,[a_i,b_i]}
\end{align}
for all $i,j\in[0:L]$ with $i<j$ and all $f\in C[a_i,b_i]$.
\end{corollary}
\begin{proof}
Let $\varrho=\varrho_0\in\bbbr_{>1}$, let $\sigma\in\bbbr_{>1}$
be as in \cref{co:nested_discs}, let $q\in(1/\sigma,1]$ and
$\Cin$ be chosen as in \cref{th:interpolation_error}.
Let $\Cos$ be the constant of \cref{le:bounded_oscillations}
with $\sigma\varrho$ instead of $\varrho$.
Let $p\in(q,1]$ and
\begin{equation*}
  \alpha_0 := \left\lceil \frac{\log(L)}{\log(p)-\log(q)} \right\rceil,
\end{equation*}
just as in \cref{rm:stability_condition}, and assume
$\alpha\geq\alpha_0$.

Let $i\in[0:L]$, let $f\in C[a_i,b_i]$ and
$\pi := \mathfrak{I}_{[a_{i+1},b_{i+1}],m_{i+1}}[\mathfrak{E}_{-c_{i+1}}[f]]$.
We have $\|\pi\|_{\infty,[a_{i+1},b_{i+1}]}\leq\Lambda_{m_{i+1}}
\|f\|_{\infty,[a_{i+1},b_{i+1}]}$, and $\hat\pi:=\mathfrak{E}_{c_{i+1}}[\pi]$
is holomorphic in the entire complex plane.
Using \cref{eq:oscillatory_stable} and
$\|\mathfrak{E}_c[f]\|_{[a_\ell,b_\ell],1}=\|f\|_{[a_\ell,b_\ell],1}$
for all $\ell\in[0:L]$, $c\in\bbbr$, $f\in C[a_\ell,b_\ell]$, we can
apply Theorem~\ref{th:interpolation_error} with $\tau$ equal to $\varrho$
and $\varrho$ equal to $1$ to get
\begin{align*}
  \|\hat\pi - &\mathfrak{I}_{j,i+1}[\hat\pi]\|_{\infty,[a_j,b_j]}
   = \|\hat\pi - \mathfrak{I}_{j,i+1}[\hat\pi]\|_{[a_j,b_j],1}\\
  &\leq \|\hat\pi - \mathfrak{I}_{j-1,i+1}[\hat\pi]\|_{[a_j,b_j],1}
      + \Cin q^\alpha \varrho^{-\alpha}
            \|\mathfrak{E}_{-c_j} \mathfrak{I}_{j-1,i+1}[\hat\pi]
            \|_{[a_j,b_j],\sigma\varrho}\\
  &\leq \|\hat\pi - \mathfrak{I}_{j-1,i+1}[\hat\pi]\|_{[a_j,b_j],1}
      + \Cos \Cin q^\alpha \varrho^{-\alpha}
            \|\mathfrak{E}_{-c_{i+1}} \mathfrak{I}_{j-1,i+1}[\hat\pi]
            \|_{[a_j,b_j],\sigma\varrho}\\
  &\leq \|\hat\pi - \mathfrak{I}_{j-1,i+1}[\hat\pi]\|_{[a_j,b_j],1}
      + \Cos \Cin q^\alpha \varrho^{-\alpha}
            \exp(\Cos^2 \Cin) \|\mathfrak{E}_{-c_{i+1}} \hat\pi\|_{[a_{i+1},b_{i+1}],\varrho}\\
  &\leq \|\hat\pi - \mathfrak{I}_{j-1,i+1}[\hat\pi]\|_{[a_j,b_j],1}
      + \Cos \Cin q^\alpha \varrho^{-\alpha}
            \exp(\Cos^2 \Cin) \|\pi\|_{[a_{i+1},b_{i+1}],\varrho}
\end{align*}
for all $j\in[i+2,L]$.
A straightforward induction yields
\begin{align*}
  \|\hat\pi - \mathfrak{I}_{j,i+1}[\hat\pi]\|_{\infty,[a_j,b_j]}
  &\leq \Cos \Cin \exp(\Cos^2 \Cin)
        (j-i-1) q^\alpha \varrho^{-\alpha} \|\pi\|_{[a_{i+1},b_{i+1}],\varrho}\\
  &\leq \Cos \Cin \exp(\Cos^2 \Cin)
        p^\alpha \varrho^{-\alpha} \|\pi\|_{[a_{i+1},b_{i+1}],\varrho}
\end{align*}
for all $j\in[i+1:L]$.
With \cref{le:bernstein_inequality} and \cref{eq:lebesgue} we conclude
\begin{align*}
  \|\mathfrak{I}_{j,i}[f]\|_{\infty,[a_j,b_j]}
  &= \|\mathfrak{I}_{j,i+1}[\hat\pi]\|_{\infty,[a_j,b_j]}
   \leq \|\hat\pi\|_{\infty,[a_j,b_j]}
       + \|\hat\pi - \mathfrak{I}_{j,i+1}[\hat\pi]\|_{\infty,[a_j,b_j]}\\
  &\leq \|\pi\|_{\infty,[a_j,b_j]}
       + \Cos \Cin \exp(\Cos^2 \Cin) p^\alpha \varrho^{-\alpha}
         \|\pi\|_{[a_{i+1},b_{i+1}],\varrho}\\
  &\leq \|\pi\|_{\infty,[a_j,b_j]}
       + \Cos \Cin \exp(\Cos^2 \Cin) p^\alpha
         \|\pi\|_{\infty,[a_{i+1},b_{i+1}]}\\
  &\leq (1 + \Cos \Cin \exp(\Cos^2 \Cin) p^\alpha) \Lambda_{m_{i+1}}
         \|f\|_{\infty,[a_i,b_i]}.
\end{align*}
Choosing $\Cst := 1 + \Cos \Cin \exp(\Cos^2 \Cin) p^\alpha$ completes
the proof.
\end{proof}

\bibliographystyle{plain}
\bibliography{hmatrix}

\end{document}